\documentclass[12pt]{article}
\usepackage[centertags]{amsmath}
\usepackage{amsfonts}
\usepackage{amssymb}
\usepackage{amsthm}
\usepackage{newlfont}
\usepackage{graphicx}
\usepackage{t1enc}
\usepackage[latin1]{inputenc}
\usepackage[english]{babel}
\usepackage{hyperref}
\newlength{\defbaselineskip}
\newcommand{\setlinespacing}[1]%
           {\setlength{\baselineskip}{#1 \defbaselineskip}}

\newcommand{\dd }{\displaystyle}

\newcommand{\f }{\\[0.1cm]}
\newcommand{\s }{\\[0.2cm]}

\newcommand{\h }{\hspace*{18pt}}
\newcommand{\bt}{\beta}
\newcommand{\al}{\alpha}
\newcommand{\G}{\Gamma}

\newcommand{\g}{\gamma}

\newcommand{\la}{\lambda}
\newcommand{\q}{\quad}
\newcommand{\qq}{\qquad}

\newcommand{\bq }{\begin{equation}}
\newcommand{\eq }{\end{equation}}
\newcommand{\bbb }{\begin{eqnarray}}
\newcommand{\eee }{\end{eqnarray}}
\newcommand{\bb }{\begin{eqnarray*}}
\newcommand{\ee }{\end{eqnarray*}}
\newcommand{\ed }{\end{document}}
\theoremstyle{plain}
\newtheorem{thm}{Theorem}
\newtheorem{lem}{Lemma}

\newtheorem{cor} {Corollary}

\theoremstyle{definition}

\theoremstyle{example}

\begin{document}
\begin{center}

{ {\bf Nonsymmetric Generalized Jacobi Petrov-Galerkin Algorithms for Third- and
Fifth-Order two Point Boundary Value Problems}}\\[0.5cm]
{\textbf{\textsf{ E.H. Doha$^{a,\,}$\footnote{eiddoha@frcu.eun.eg},
W. M. Abd- Elhameed$^{b, a,\,}$\footnote{walee$_{-}$9@yahoo.com,
corresponding author},
 Y. H. Youssri$^{a,\,}$\footnote{youssri@sci.cu.edu.eg}}}}\\[0.2cm]
$^{a}$Department of Mathematics, Faculty of Science, Cairo
University, Giza-Egypt\\
$^{b}$Department of Mathematics,Faculty of Science,King Abdulaziz
University, Saudi Arabia
\end{center}
\hrule\[\]
\\[-0.4cm]
{\bf Abstract.} Two families of certain nonsymmetric generalized
Jacobi polynomials with negative integer indexes are used for
solving third- and fifth-order two point boundary value problems
subject to homogeneous and nonhomogeneous boundary conditions using
a dual Petrov-Galerkin method. The key
 idea behind our method is to use trial functions
 satisfying the underlying boundary conditions of the differential
 equations and the test functions satisfying the dual boundary
 conditions.The method leads to linear systems with specially structured matrices
that can be efficiently inverted. The use of generalized Jacobi
polynomials leads to simplified analysis, very efficient numerical
algorithms.
  Numerical results are presented to demonstrate the
 efficiency of our proposed algorithms.\s
\section{Introduction}
\h Spectral method, in the context of numerical schemes, was
introduced and popularized by Orszag's pioneer work in the early
seventies. The term spectral was probably originated from the fact
that the trigonometric functions $e^{ikx}$ are the eigenfunctions of
the Laplace operator with periodic boundary conditions. This fact
and the availability of Fast Fourier Transform (FFT) are two main
advantages of the Fourier spectral method. Thus, using Fourier
series to solve PDEs, with principal differential operator being the
Laplace operator (or its power) with periodic boundary conditions,
results in very attractive numerical algorithms. However, for
problems with rigid boundaries, the eigenfunctions of Laplace
operator (with non-periodic boundary conditions), although easily
available in regular domains, are no longer good candidates as basis
functions due to the Gibbs phenomenon . In such cases, it is well
known that one should use the eigenfunctions of the singular
Sturm-Liouville operator i.e., Jacobi polynomials with a suitable
pair of indexes. \s\h Standard spectral methods are capable of
providing very accurate approximations to well-behaved smooth
functions with significantly less degrees of freedom when compared
with finite difference or finite element methods (cf.
\cite{Boy2001},\cite{Can1988},\cite{GO1977}).\s \h Classical
orthogonal polynomials are used successfully  and
 extensively for the numerical solution of differential equations in
spectral and pseudospectral methods (see, for instance,
\cite{Bialecki}, \cite{Can1988,Cou1996}, \cite{DA2005}
,\cite{Heinrichs1989a} and \cite{Heinrichs1989b}).\s
 \h
The classical Jacobi polynomials\ $P_{n}^{(\alpha,\beta)}(x)$ play
 important roles in mathematical analysis and its applications (see \cite{Sze1985}). In particular,
 the Legendre, the Chebyshev, the ultraspherical polynomials have palyed important roles in spectral methods for
 partial differential equations (see, for instance, \cite{Boy2001},\cite{Fun1992}).
  It is proven that the
Jacobi polynomials are precisely the only polynomials arising as
eigenfunctions of a singular Sturm-Liouville problem, (see
\cite{Can1988}, Sec. 9.2). This class of polynomials comprises all
the polynomial solution to singular Sturm-Liouville problems on
$[-1,1]$. \s
 \h Guo et al. \cite{Guo2006} extended the definition of
 the classical Jacobi polynomials with indexes\ $\al,\bt>-1$ to
 allow\ $\al$\ and/or\ $\bt$\ to be negative integers. They showed also that the generalized
Jacobi polynomials, with indexes corresponding to the number of
boundary conditions in a given partial differential equation, are
the natural basis functions for the spectral approximation of this
equation. Moreover it is shown that the use of generalized Jacobi
polynomials not only simplified the numerical analysis for the
spectral approximations of differential equations, but also led to
very efficient numerical algorithms.\s\h Abd-Elhameed et al.
\cite{y1} and Doha et al. \cite{y} used the general parameter
generalized Jacobi polynomilas to handle third- and fifth-order
differential equations.\s\h The majority of books and research
papers dealing with the theory of ordinary differential equations,
or their practical applications to technology and physics, contain
mainly results from the theory of second-order linear differential
equations, and some results from the theory of some special linear
differential equations of higher even order. However there is only a
limited body of literature on spectral methods for dispersive,
namely, third- and fifth-order equations. This is partly due to the
fact that direct collocation methods for third- and fifth-order
boundary problems lead to
 condition numbers of high order, typically of order $N^{6}$ and $N^{10}$ respectively, where $N$\ is
 the number of retained modes. These high
condition numbers will lead to instabilities caused by rounding
errors (see, \cite{Hua} and \cite{Mer}). In this paper, we introduce
some efficient spectral algorithms for reducing these condition
numbers to be of $O(N^2)$ and $O(N^4)$ for third- and fifth-order
respectively, based on certain nonsymmetric generalized Jacobi
Petrov-Galerkin method.\s
  \h The study of odd-order equations is of interest,
for example, the third order equation is of fundamental mathematical
interest since it lacks symmetry. Also, it is of physical interest
since it contains a type of operator which appears in many commonly
occurring partial differential equations such as the Kortweg-de
Vries equation. Monographs like those of Mckelvey \cite{Mck}, which
include chapters on oscillation properties of third-order
differential equations, are exceptional. The interested reader in
applications of odd-order differential equations is referred to the
monograph by (Gregus \cite{Gre}), in which many physical and
engineering applications of third-order differential equations are
discussed [see, pp. 247-258].\s
 \h From the numerical point of view,  Abd-Elhameed \cite{Abd1}, Doha and Abd-Elhameed
 \cite{Doha and Abd-Elhameed2002,Doha and Abd-Elhameed2009}, Doha and Bhrawy \cite{Bhrawy} and
  Doha et al. \cite{Doha et al2008}
 have constructed efficient spectral-Galerkin algorithms using compact
 combinations of orthogonal polynomials for solving elliptic equations
 of the second-, fourth-, {\it 2n}th- and {\it (2n+1)th}-order in various
 situations. Recently, Doha and Abd-Elhameed \cite{Doha and Abd-Elhameed2012} have introduced a family of symmetric generalized Jacobi polynomials for solving multidimensional
 sixth-order two point boundary value problems by the Galerkin method. Also some other studies are
devoted to third- and fifth-order differential equations in finite
intervals (see, \cite{Ma1,Ma2}).\s
 \h  The main differential operator in odd-order
 differential equations
 is not symmetric, so it is convenient to use a Petrov-Galerkin
 method. The difference between Galerkin and Petrov-Galerkin methods, is
 that the test and
trial functions in Galerkin method are the same, but for
Petrov-Galerkin method, the trial functions are chosen to satisfy
the boundary conditions of the differential equation, and the test
functions are chosen to satisfy the dual boundary conditions.\s \h
In this paper we are concerned with the direct solution techniques
for third- and fifth-order elliptic equations, using the generalized
Jacobi Petrov-Galerkin method (GJPGM). Our algorithms lead to
discrete linear systems with specially structured
 matrices that can be efficiently inverted.\s
\h We organize the materials of this paper as follows. In Section 2,
we give some
 properties of classical and generalized Jacobi polynomials. In
 Sections 3 and 4, we are interested in using GJPGM to solve third- and
fifth-order linear differential equation with constant coefficients
subject to homogenous boundary conditions. In Section 5, we study
the structure of the coefficient matrices in the systems resulted
from applying GJPGM. In Section 6, we are interested in using GJPGM
to solve third- and fifth-order linear differential equation with
constant coefficients subject to nonhomogeneous boundary conditions.
In Section 7, the condition numbers of the systems resulted from
applying GJPGM are discussed. In Section 8, we discuss some
numerical results. Some concluding remarks are given in Section 9.
  \section{Some  properties of Classical and generalized Jacobi polynomials}
   \subsection{Classical Jacobi polynomials}
 \h The classical Jacobi polynomials associated with the real parameters
$(\alpha >-1,\,\beta >-1)$ (see, \cite{Abramowich and Stegun1970},
\cite{Andrews et al} and \cite{Sze1985}),
 are a sequence of polynomials $%
P_{n}^{(\alpha ,\beta )}(x), x\in(-1,1) (n=0,1,2,...)$, each
respectively of degree $n$. For our present purposes, it is more
convenient to introduce the normalized orthogonal polynomials
$R_{n}^{(\alpha ,\beta )}(x)=\dd\frac{P_{n}^{(\alpha ,\beta
)}(x)}{P_{n}^{(\alpha ,\beta )}(1)}.$ This means that
$R_{n}^{(\alpha ,\beta
)}(x)=\dd\frac{n!\,\Gamma(\al+1)}{\Gamma(n+\al+1)}\, P_{n}^{(\alpha
,\beta )}(x)$. In such case
$R_{n}^{(\alpha-\frac12,\alpha-\frac12)}(x)$ is identical to the
ultraspherical polynomials $C^{(\alpha)}_{n}(x)$, and the
polynomials $R_{n}^{(\alpha ,\beta )}(x)$ may be generated using the
recurrence relation
  \begin{eqnarray*}&2&(n+\lambda)(n+\alpha+1)(2n+\lambda-1)R_{n+1}^{(\alpha,\beta)}(x)=
 (2n+\lambda-1)_{3}\, x\, R_{n}^{(\alpha,\beta)}(x)
 \\&+&(\alpha^2-\beta^2)(2n+\lambda)
 R_{n}^{(\alpha,\beta)}(x)-2n(n+\beta)(2n+\lambda+1)
 R_{n-1}^{(\alpha,\beta)}(x),\quad n=1,2,\dots,\end{eqnarray*}
starting from $R_{0}^{(\alpha ,\beta )}(x)=1$ and $R_{1}^{(\alpha
,\beta )}(x)=\dd\frac{1}{2(\alpha+1)}[\alpha -\beta +(\lambda
+1)x]$, or obtained from Rodrigue's formula\\
$$ R_{n}^{(\alpha ,\beta )}(x)=\bigg(\frac{-1}{2}\bigg)^{n}\, \dd\frac{\Gamma(\alpha+1)}
 {\Gamma(n+\alpha+1)}(1-x)^{-\alpha }(1+x)^{-\beta
}D^{n}\bigg[(1-x)^{\alpha +n}(1+x)^{\beta +n}\bigg], $$
 where
\[\lambda =\alpha +\beta+1,\qq(a)_k=\frac{{\G(a+k)}}{\G{(a)}},\qq D=\frac{d}{dx},\]
and satisfy the orthogonality relation
 \bq
 \label{ortho}
\int_{-1}^{1}(1-x)^{\alpha}(1+x)^{\beta }\ R_{m}^{(\alpha
,\beta )}\ (x)R_{n}^{(\alpha ,\beta )}(x)\,dx =\left\{
  \begin{array}{ll}
  0, \ \qquad\   m \neq n, \s
    h^{\alpha,\beta}_{n}, \ \ \quad m=n,
  \end{array}\right.
  \eq
where
\[
 h^{\alpha,\beta}_{n}=\frac{2^{\lambda}\ n!\  \Gamma (n+\beta +1)\, \left[
 \Gamma(\alpha+1)\right]^2}
 {(2n+\lambda)\, \Gamma(n+\lambda)\, \Gamma
(n+\alpha +1)}.
\]
These polynomials are eigenfunctions of the following singular
Sturm-Liouville equation:\s \bq(1-x^{2})\,\phi^{\prime \prime
}(x)+[\,\beta -\alpha -(\lambda+1)\,x]\,\phi^{\prime }(x)
+n(n+\lambda)\,\phi(x)=0.\nonumber\eq
  The following relations will be of important use later.\s
 \bq
R_{k}^{(\alpha,\beta)}(x)=\dd\frac{1}{k+1}\left[(k+\alpha+1)\,
R^{(\alpha,\beta-1)}_{k+1}(x)
 -\alpha \, R^{(\alpha-1,\beta)}_{k+1}(x)\right],\label{shift1}
 \eq
 \bq
R_{k}^{(\alpha,\beta)}(x)=\dd\frac{1}{k+\alpha+\beta}\left[(k+\beta)\,
R^{(\alpha,\beta-1)}_{k}(x)
 +\alpha \, R^{(\alpha-1,\beta)}_{k}(x)\right],\label{shift2}
 \eq
 \bq
 (1-x)R_{k}^{(\alpha+1,\beta)}(x)=\dd\frac{2(\al+1)}{2k+\alpha+\beta+2}\left[R^{(\alpha,\beta)}_{k}(x)
 -R^{(\alpha,\beta)}_{k+1}(x)\right],\label{shift3}
 \eq
\bq
 \begin{split}
 \label{6}
& (1-x^2)\ R_{k-1}^{(\alpha+1,\beta+1)}(x)=\dd\frac{4(\alpha+1)}{(2k+\lambda-1)_{3}}\
 \left[(k+\beta)(2k+\lambda+1)R_{k-1}^{(\alpha,\beta)}(x)\right.\\
 &\left.-(k+\alpha+1)(2k+\lambda-1)
R_{k+1}^{(\alpha,\beta)}(x)+(\alpha-\beta)(2k+\lambda)\,
R_{k}^{(\alpha,\beta)}(x)\right],
 \end{split}
\eq
 \bq
 D^{q}R_{k}^{(\alpha,\beta)}(x)=\dd\frac{(k-q+1)_{q}\ (k+\la)_{q}}{2^{q}\, (\alpha+1)^q}\
R_{k-q}^{(\alpha+q,\beta+q)}(x).\label{dershift}
 \eq
Now, the following theorem is needed in the sequel.
 \begin{thm}
 \label{thm2}
 The qth derivative of the normalized Jacobi polynomial
 $R_n^{(\al,\bt)}(x)$ is given explicitly by
 \begin{equation}
 D^{q}R_{n}^{(\alpha ,\beta
)}(x)=(n+\lambda)_{q}\ 2^{-q}\, n! \
\sum_{i=0}^{n-q}C_{n-q,i}(\alpha +q,\beta +q,\alpha ,\beta )\
R_{i}^{(\alpha ,\beta )}(x),
 \nonumber\end{equation}
where \s
 $C_{n-q,i}(\alpha +q, \beta +q,\alpha ,\beta
)=\dd\frac{(n+q+\lambda)_{i} \ (i+q+\alpha +1)_{n-i-q}\ \Gamma
(i+\lambda)}{(n-i-q)!\ \Gamma (2i+\lambda)\, i!\,
(i+\alpha+1)_{n-i}}$\s
 \hspace*{140pt}$
 \times _{3}F_{2}\left(
\begin{array}{lll}
-n+q+i,\quad n+i+q+\lambda,\ \ i+\alpha
+1\\\hspace*{200pt};1\\i+q+\alpha +1,\quad 2i+\lambda +1
\end{array}
\right) .$
 \end{thm}
 \noindent
 (For the proof of Theorem \ref{thm2}, see Doha \cite{Doha20021}).\s
 \subsection{Generalized Jacobi polynomials}
\h  Following \cite{Guo2006}, we can define
  a family of generalized Jacobi polynomials/functions with indexes $\alpha$, $\beta\in
  \mathbb{R}$.\s
\h Let $w^{\alpha,\beta}(x)=(1-x)^{\alpha}(1+x)^\beta.$  We denote
by $L^{2}_{w^{\alpha,\beta}}(-1,1)$ the weighted $L^{2}$\  space
with inner product:
$$(u,v)_{w^{\alpha,\beta}}(x):=\int_{I}u(x)v(x)
w^{\alpha,\beta}(x)\, dx,$$
 and the associated norm $\|u\|_{w^{\alpha,\beta}}=(u,u)^{\frac12}_{w^{\alpha,\beta}}$.
We are interested in defining Jacobi polynomials with indexes
$\alpha$ and/or $\beta\le -1$, referred hereafter as generalized
Jacobi polynomials (GJPs), in such a way that they satisfy some
selected properties that are essentially relevant to spectral
approximations. In this work, we shall restrict our attention to the
cases when $\alpha$ and $\beta$ are negative integers.\s
 \h  Let $\ell,m\in \mathbb{Z}$
(the set of all integers),
\begin{equation}
J_{k}^{(\ell,m)}(x)=
  \begin{cases}
  (1-x)^{-\ell}\, (1+x)^{-m}\, R^{(-\ell,-m)}_{k-k_{0}}(x),\quad &k_{0}=-(\ell+m), \  \ell,m\le -1, \\
  (1-x)^{-\ell}\, R^{(-\ell,m)}_{k-k_{0}}(x),\quad &k_{0}=-\ell,\  \ell\le -1,\, m>-1, \\
   (1+x)^{-m}\, R^{(\ell,-m)}_{k-k_{0}}(x),\quad &k_{0}=-m ,\  \ell> -1,\, m\leq-1, \\
    R^{(\ell,m)}_{k-k_{0}}(x), \quad &k_{0}=0,\  \ell,m>-1. \\
  \end{cases}
\nonumber\end{equation}
 \h An important property of the GJPs is that for \ $\ell ,m,\in
\mathbb{Z}$\ and\ $\ell,m\ge 1$,\\
  $$D^{i}J^{(-\ell,-m)}_{k}(1)=0,\quad i=0,1,\ldots,\ell-1;$$
  $$D^{j}J^{(-\ell,-m)}_{k}(-1)=0,\quad j=0,1,\ldots,m-1.$$
It is not difficult to verify that \bq
J_{k}^{(-2,-1)}(x)=\frac{4}{(k-1)(2k-3)}\,
\left[L_{k-3}(x)-\frac{2k-3}{2k-1}\, L_{k-2}(x)-L_{k-1}(x)+
\frac{2k-3}{2k-1}\, L_{k}(x)\right],\nonumber\eq\bq\hspace*{200pt}
k\ge3,\nonumber\eq \bq J_{k}^{(-1,-2)}(x)=\frac{2}{2k-3}\,
\left[L_{k-3}(x)+\frac{2k-3}{2k-1}\, L_{k-2}(x)-L_{k-1}(x)-
\frac{2k-3}{2k-1}\, L_{k}(x)\right],\nonumber\eq\bq\hspace*{200pt}
k\ge 3,\nonumber\eq \bq
J_{k}^{(-3,-2)}(x)=\frac{24}{(2k-5)(2k-7)(k-2)}\left[L_{k-5}(x)-\frac{(2k-7)}{2k-3}\,
 L_{k-4}(x)
 -\frac{2(2k-5)}{2k-3}\, L_{k-3}(x)\right.\nonumber\eq
\bq\hspace*{15pt}\left.\qquad\quad+\dd\frac{2(2k-7)}{2k-1}\,
L_{k-2}(x)+\dd\frac{2k-7}{2k-3}\, L_{k-1}(x)-
 \dd\frac{(2k-5)(2k-7)}{(2k-1)(2k-3)}\, L_{k}(x)\right],\ k\ge
 5,\nonumber\eq
\bq
J_{k}^{(-2,-3)}(x)=\frac{8}{(2k-5)(2k-7)}\left[L_{k-5}(x)+\frac{2k-7}{2k-3}\,
 L_{k-4}(x)
 -\frac{2(2k-5)}{2k-3}\, L_{k-3}(x)\right.\nonumber\eq
\bq \left.\qquad\quad-\dd\frac{2(2k-7)}{2k-1}\,
L_{k-2}(x)+\dd\frac{2k-7}{2k-3}\, L_{k-1}(x)+
 \dd\frac{(2k-5)(2k-7)}{(2k-1)(2k-3)}\, L_{k}(x)\right],\  k\ge 5,\nonumber\eq
where\ $L_{k}(x)$\ is the Legendre polynomial of the {\it k}th
degree. $\{J_{k}^{(-\ell,-m)}(x)\}$\ are natural
candidates as basis
  functions for PDFs with the following boundary conditions:
  $$D^{i}u(1)=a_{i},\quad i=0,1,\ldots,\ell-1;$$
  $$D^{j} u(-1)=b_{j},\quad j=0,1,\ldots,m-1.$$

\section{ Dual Petrov-Galerkin algorithms for third-order elliptic linear differential equations}
\h
 We are interested in using the generalized Jacobi-Petrov-Galerkin method
 to solve the following third-order elliptic linear differential equation
 \bq\,u^{(3)}(x)\,-\al_1\, u^{(2)}(x)-\bt_1\, u^{(1)}(x)+\g_1\, u(x)=f(x),\qquad x\in (-1,1),\label{900}
 \eq
  subject to the homogeneous boundary conditions
 \begin{equation}\label{1000}
u(\pm 1)=u^{(1)}(1)=0.
\end{equation}
We define the space
$$V=\{{u\in H^{(2)}(I):u(\pm
1)=u^{(1)}(1)=0}\},$$\,and its dual space
$$
V^*=\{{u\in H^{(2)}(I):u(\pm 1)=u^{(1)}(-1)=0}\}.$$ where \bq
H^{(2)}(I)=\{u:\|u\|_{2,w^{\al,\bt}}<\infty\},\,
\|u\|_{2,w^{\al,\bt}}=\left(\dd\sum_{k=0}^2\|\partial_x^ku\|^2_{w^{\al+k,\bt+k}}\right)^{\frac12}\nonumber\eq
 Let\ $P_{N}$\ be the space of all polynomials of degree less than
 or  equal to $N$. Setting\ $V_{N}=V\cap P_{N}$\ and\ $V_{N}^*=V^*\cap
P_{N}$. We observe that:
$$
V_{N}={\mathrm {span}}\{J^{(-2,-1)}_{3}(x),\,
J^{(-2,-1)}_{4}(x),...,J^{(-2,-1)}_{N}(x)\},
 $$
$$ V^*_{N}
=\mathrm{
span}\{J^{(-1,-2)}_{3}(x),J^{(-1,-2)}_{4}(x),...,J^{(-1,-2)}_{N}(x)\}.
$$
The dual Petrov-Galerkin approximation of (\ref{900})-(\ref{1000})
is to find $u_{N}\in V_{N}$\ such that
\bq
\begin{split}
\left(D^{3}u_{N}(x),v(x)\right)-\al_1\left(D^{2}u_{N}(x),v(x)\right)
-\bt_1\left(Du_{N}(x),v(x)\right)\\
+\g_1\left(u_{N}(x),v(x)\right)=(f(x),v(x)),\q \forall v\in
V^{*}_N.
\end{split}
\eq

 {\subsection{The choice of basis functions}}
\h   We can construct suitable basis functions and their dual basis
by setting \bq \phi_{k}(x)=J_{k+3}^{(-2,-1)}(x)=(1-x^2)(1-x)\,
R_{k}^{(2,1)}(x),\quad k=0,1,\ldots,N-3,
  \nonumber\eq
 \bq
 \psi_{k}(x)=J_{k+3}^{(-1,-2)}(x)=(1-x^2)(1+x)\,
R_{k}^{(1,2)}(x),\quad k=0,1,\ldots,N-3.
  \nonumber\eq\h It is obvious that $\{\phi_{k}(x)\}$\ and\ $\{\psi_{k}(x)\}$\ are
linearly independent. Therefore we have
 $$V_{N}={\mathrm {span}}\{\phi_{k}(x):k=0,1,2,\ldots,N-3\},$$
 and
 $$V_{N}^{*}={\mathrm {span}}\{\psi_{k}(x):k=0,1,2,\ldots,N-3\}.$$
Now we state and prove the following two lemmas.
\begin{lem}
 \label{dersy}
 \begin{equation}
D^{3}J_{k+3}^{(-2,-1)}(x)=\dd2(k+1)(k+3)
R_{k}^{(1,2)}(x).\label{diff3}
 \end{equation}
 \end{lem}
 \noindent
 \begin{proof}
 By using Leibnitz's rule, we have
 \begin{eqnarray}
D^{3}J_{k+3}^{(-2,-1)}(x)=&&(1-x^{2})(1-x)D^{3}R_{k}^{(2,1)}(x)
 +3(3x^{2}-2x-1)D^{2}R_{k}^{(2,1)}(x)\nonumber\\&&+6(-1+3x)DR_{k}^{(2,1)}(x)+6R_{k}^{(2,1)}(x).\label{26}\nonumber
 \end{eqnarray}
Making use of the relation
\bq\label{Dp3}(1-x^{2})(1-x)D^{3}R_{k}^{(2,1)}(x)=(1+6x-7x^{2})D^{2}R_{k}^{(2,1)}(x)
+(k-1)(k+5)(x-1)DR_{k}^{(2,1)}(x),\nonumber\eq we obtain \bq
D^{3}J_{k+3}^{(-2,-1)}(x)=2(x^2-1)D^2R_{k}^{(2,1)}(x)+\big[(k-1)(k+5)(x-1)+6(3x-1)\big]
DR_{k}^{(2,1)}(x)+6R_{k}^{(2,1)}(x),\nonumber\eq which in turn with
equation (2), and after some manipulation, yields
 \bq
D^{3}J_{k+3}^{(-2,-1)}(x)=-(k+1)(k+3)\big[(1-x)DR_{k}^{(2,1)}(x)-2R_{k}^{(2,1)}(x)\big].
\nonumber\eq Making use of the two relations (\ref{shift3}) and
(\ref{dershift}), we have \bq
D^{3}J_{k+3}^{(-2,-1)}(x)=\frac16(k+1)(k+3)\big[k(k+4)(x-1)R_{k-1}^{(3,2)}(x)+12R_{k}^{(2,1)}(x)\big].\nonumber\eq
Finally, with the aid of the two relations (\ref{shift1}) and
(\ref{shift2}), and after some manipulation, we get \bq
D^{3}J_{k+3}^{(-2,-1)}(x)=\dd2(k+1)(k+3)
R_{k}^{(1,2)}(x).\nonumber\eq
\end{proof}
\begin{lem}
\label{D21}
\begin{eqnarray}
\qquad
D^{2}J_{k+3}^{(-2,-1)}(x)&=&\frac{2(k+3)_{2}}{(2k+5)}R_{k+1}^{(1,2)}(x)-
\frac{(k+1)(k+3)}{(k+\frac{3}{2})_{2}}R_{k}^{(1,2)}(x)\nonumber\\&
&-
\frac{2(k)_{2}}{(2k+3)}R_{k-1}^{(1,2)}(x),\label{diff2}\nonumber\\
\qquad
DJ_{k+3}^{(-2,-1)}(x)&=&\frac{(k+3)_{3}}{2(k+2)(k+\frac{5}{2})_{2}}R_{k+2}^{(1,2)}(x)-
\frac{(k+3)_{2}}{(k+\frac{3}{2})_{3}}R_{k+1}^{(1,2)}(x)-
\frac{(k+1)(k+3)}{(k+\frac{3}{2})_{2}}R_{k}^{(1,2)}(x)\nonumber\\&
&+ \frac{(k)_{2}}{(k+\frac{1}{2})_{3}}R_{k-1}^{(1,2)}(x)+
\frac{(k-1)_{3}}{2(k+2)(k+\frac{1}{2})_{2}}R_{k-2}^{(1,2)}(x),\label{diff1}\nonumber\\
 J_{k+3}^{(-2,-1)}(x)
&=&\frac{(k+4)_{3}}{4(k+2)(k+\frac{5}{2})_{3}}R_{k+3}^{(1,2)}(x)-
\frac{3(k+3)_{3}}{4(k+2)(k+\frac{3}{2})_{4}}R_{k+2}^{(1,2)}(x)-
\frac{3(k+3)_{2}}{4(k+\frac{3}{2})_{3}}R_{k+1}^{(1,2)}(x)\nonumber\\&
&+ \frac{3(k+1)(k+3)}{2(k+\frac{1}{2})_{4}}R_{k}^{(1,2)}(x)+
\frac{3(k)_{2}}{4(k+\frac{1}{2})_{3}}R_{k-1}^{(1,2)}(x)\nonumber\\&
&- \frac{3(k-1)_{3}}{4(k+2)(k-\frac{1}{2})_{4}}R_{k-2}^{(1,2)}(x)-
\frac{(k-2)_{3}}{4(k+2)(k-\frac{1}{2})_{3}}R_{k-3}^{(1,2)}(x).\label{diff0}\nonumber\end{eqnarray}
 \end{lem}
 \begin{proof}
The proof of Lemma \ref{D21} is rather lengthy and it can
be accomplished by following the same procedure used in the proof of
Lemma \ref{dersy}.
\end{proof}
 \h Now, based on the two Lemmas \ref{dersy} and \ref{D21}, the following
theorem can be obtained.
\begin{thm}\label{thm3}
 We have, for arbitrary constants\ $a_{k}$,
  \bq
  D^{3}\left[\dd\sum_{k=0}^{N-3}\ a_{k} J_{k+3}^{(-2,-1)}(x)\right]=\dd\sum_{k=0}^{N-3}b_{k}\
R_{k}^{(1,2)}(x),\label{D3J}
  \eq
  where
  \bq
  b_{k}=2(k+1)(k+3)a_{k}.\label{bk}
  \eq Moreover, if
\bq D^{2}\left[\dd\sum_{k=0}^{N-3}\ a_{k}
J_{k+3}^{(-2,-1)}(x)\right]=\dd\sum_{k=0}^{N-2}e_{k,2}
R_{k}^{(1,2)}(x),
  \label{35}\eq
  then
  \bq
e_{k,2}=a_{k-1}\,\al^{(2)}_{k-1}+a_k\,\bt^{(2)}_k+a_{k+1}\,\g^{(2)}_{k+1},\label{e2}\eq
where
\begin{alignat*}{3} &\al^{(2)}_k=\frac{2(k+3)_{2}}{(2k+5)},\quad
&&\bt^{(2)}_k=-\frac{(k+1)(k+3)}{(k+\frac{3}{2})_{2}},\quad
 &&\g^{(2)}_k=-\frac{2(k)_{2}}{(2k+3)}.\f
 \end{alignat*}
Also, if\bq D\left[\dd\sum_{k=0}^{N-3}\ a_{k}
J_{k+3}^{(-2,-1)}(x)\right]=\dd\sum_{k=0}^{N-1}e_{k,1}\
R_{k}^{(1,2)}(x),
  \label{37}\eq then \bq
e_{k,1}=a_{k-2}\,\al^{(1)}_{k-2}+a_{k-1}\,\bt^{(1)}_{k-1}+a_{k}\,\g^{(1)}_{k}+a_{k+1}\,
  \delta^{(1)}_{k+1}+a_{k+2}\,\mu^{(1)}_{k+2},\label{e1}\eq where
\begin{alignat*}{3} &\al^{(1)}_{k}=\frac{(k+3)_{3}}{2(k+2)(k+\frac{5}{2})_{2}},\quad
&&\bt^{(1)}_{k}=-\frac{(k+3)_{2}}{(k+\frac{3}{2})_{3}},\quad
 &&\g^{(1)}_{k}=-
\frac{(k+1)(k+3)}{(k+\frac{3}{2})_{2}},\f &\delta^{(1)}_{k}=
\frac{(k)_{2}}{(k+\frac{1}{2})_{3}},\quad &&\mu^{(1)}_{k}=
\frac{(k-1)_{3}}{2(k+2)(k+\frac{1}{2})_{2}}.\f
 \end{alignat*}
Finally, if\bq \dd\sum_{k=0}^{N-3}\ a_{k}
J_{k+3}^{(-2,-1)}(x)=\dd\sum_{k=0}^{N}e_{k,0}\
R_{k}^{(1,2)}(x),\label{39}\eq then \bq
e_{k,0}=a_{k-3}\,\al^{(0)}_{k-3}+a_{k-2}\,\bt^{(0)}_{k-2}+a_{k-1}\,\g^{(0)}_{k-1}+a_{k}\,
  \delta^{(0)}_{k}+a_{k+1}\,\mu^{(0)}_{k+1}+a_{k+2}\,\eta^{(0)}_{k+2}+a_{k+3}\,\zeta^{(0)}_{k+3},\label{sumJ}
  \eq
  where
\begin{alignat*}{3}&\al^{(0)}_{k}=\frac{(k+4)_{3}}{4(k+2)(k+\frac{5}{2})_{3}},\quad
&&\bt^{(0)}_{k}=-\frac{3(k+3)_{3}}{4(k+2)(k+\frac{3}{2})_{4}},\quad
 &&\g^{(0)}_{k}=-
\frac{3(k+3)_{2}}{4(k+\frac{3}{2})_{3}},\f &\delta^{(0)}_{k}=-
 \frac{3(k+1)(k+3)}{2(k+\frac{1}{2})_{4}},\quad
&&\mu^{(0)}_{k}=\frac{3(k)_{2}}{4(k+\frac{1}{2})_{3}},
 &&\eta^{(0)}_{k}=-
\frac{3(k-1)_{3}}{4(k+2)(k-\frac{1}{2})_{4}},\quad\f
&\zeta^{(0)}_{k}=-
\frac{(k-2)_{3}}{4(k+2)(k-\frac{1}{2})_{3}}.\f\end{alignat*}
\end{thm}
 The application of Petrov-Galerkin
method to equation (\ref{900}), gives
  \bq
 \left(D^{3}\, u_{N}(x)-\al_1\,D^{2}u_{N}-\bt_1\,D u_{N}
 +\g_1\,u_{N},\psi_{k}(x)\right)
 =\left(f(x),\psi_{k}(x)\right),
\label{gal}
  \eq
  where
  $$\displaystyle  u_N(x)=\dd\sum_{k=0}^{N-3}a_k\ \phi_{k}(x),\quad \phi
_k(x)=J_{k+3}^{(-2,-1)}(x), \quad \psi
_k(x)=J_{k+3}^{(-1,-2)}(x),\quad k=0,1,\ldots ,N-3.$$
 Substitution of formulae (\ref{D3J}), (\ref{35}), (\ref{37}) and (\ref{39}) into (\ref{gal}) yields\s
\bq\bigg(\dd\sum_{j=0}^{N-3}b_{j}\,R_{j}^{(1,2)}(x)-\dd
\al_1\dd\sum_{j=0}^{N-2} e_{j,2}\,R_{j}^{(1,2)}(x)-\dd
\bt_1\dd\sum_{j=0}^{N-1}
e_{j,1}\,R_{j}^{(1,2)}(x)\nonumber\eq\bq+\dd \g_1\dd\sum_{j=0}^{N}
e_{j,0}\, R_{j}^{(1,2)}(x),J_{k+3}^{(-1,-2)}(x)\bigg)
 =\left(f,J_{k+3}^{(-1,-2)}(x)\right),\label{galerkin2}
 \eq
  where $b_k$ and $e_{k,2-q}, 0\leqslant q \leqslant 2$ are given by
  (\ref{bk}), (\ref{e2}), (\ref{e1}) and (\ref{sumJ}) respectively.\s
Eq. (\ref{galerkin2}) is equivalent
  to
  \bq\bigg(\dd\sum_{j=0}^{N-3}b_{j}\,R_{j}^{(1,2)}(x)-\dd
\al_1\dd\sum_{j=0}^{N-2} e_{j,2}\,R_{j}^{(1,2)}(x)-\dd
\bt_1\dd\sum_{j=0}^{N-1}
e_{j,1}\,R_{j}^{(1,2)}(x)\nonumber\eq\bq+\dd \g_1\dd\sum_{j=0}^{N}
e_{j,0}\, R_{j}^{(1,2)}(x),R_k^{(1,2)}(x)\bigg)_w
 =\left(f,R_k^{(1,2)}(x)\right)_w,
 \nonumber\eq where $w=(1-x^2)(1+x).$ Making use of the orthogonality
relation (\ref{ortho}), it is not difficult to show that Eq.
(\ref{galerkin2}) is equivalent to
\begin{equation}
f_{k}=\left(b_{k}-\dd \al_1\, e_{k,2}-\dd \bt_1\, e_{k,1}+\dd \g_1\,
e_{k,0}\right)\, h_{k};\ k=0,1,\ldots N-3.\label{system1}
\end{equation}
 where
 \bq
 f_{k}=\bigg(f,R_{k}^{(1,2)}(x)\bigg)_{w}.
 \nonumber\eq
This linear system may be put in the form \s
\begin{equation}
\left(b^1_{k}-\dd \al_1\, e_{k,2}-\dd \bt_1\, e_{k,1}+\dd \g_1\,
e_{k,0}\right)=f^{*}_k,\ k=0,1,\ldots N-3,\label{system}
\end{equation}
where\ \bq f^{*}_k=\dd\frac{f_{k}} {h_{k}^{1,2}},\quad
h_{k}^{1,2}=\dd\frac{8}{(k+1)(k+2)(k+3)}.
\label{righthand}\nonumber\eq
 which may be written in the matrix form
\begin{equation}\
(B_1+\dd \al_1\, E_{2}+\dd \bt_1\, E_{1}+\dd \g_1\, E_{0})\,
\mathbf{a}=\mathbf{f^*},\label{system1}
\end{equation}
where
\begin{center}
$\mathbf{a}$=$(a_0,a_1,\ldots ,a_{N-3})^T$,\
$\mathbf{f^{*}}$=$(f^{*}_0,f^{*}_1, \ldots ,f^{*}_{N-3})^T,$
\end{center}
and the nonzero elements of the matrices $B,E_{2},E_{1}$ and
$E_{0}$\ are given explicitly in the following theorem.
\begin{thm}\label{thm4} The nonzero elements\
$(b^1_{kj})$ and $(e^{i, 1}_{kj}),\ 0\le
 i\le 2\
{{\text for}} \ 0\leq k,j\leq N-3$ are \ given \ as \ follows:
\begin{align*}
&b^1_{kk}=2(k+1)(k+3),&&
 e^{2,1}_{k,k+1}=\dd\frac{2(k+1)(k+2)}{2k+5},\\
 & e^{2,1}_{k+1,k}=\dd\frac{-2(k+3)(k+4)}{2k+5},
&&e^{2,1}_{kk}=\frac{4\, (k+1)(k+3)}{(2 k+3)(2k+5)},\\
&e^{1,1}_{kk}=\frac{4\, (k+1)(k+3)}{(2 k+3)(2k+5)},
  &&\displaystyle
e^{1,1}_{k,k+1}=\frac{-8(k+1)(k+2)}{(2k+3)(2 k+5)(2 k+7)},\\
 &\displaystyle
e^{1,1}_{k,k+2}=\frac{-2(k+1)(k+2)(k+3)}{(k+4)(2k+5)(2 k+7)},
&&e^{1,1}_{k+1,k}=\frac{8(k+3)(k+4)}{(2k+3)(2k+5)(2 k+7)},\\
&e^{1,1}_{k+2,k}=\frac{-2(k+3)(k+4)(k+5)}{(k+2)(2k+5)(2 k+7)},
 &&e^{0,1}_{kk}=\dd\frac{3(k+1)(k+3)}{2(k+\frac12)_{4}},\\
 &e^{0,1}_{k,k+1}=\frac{3(k+1)_{2}}{4(k+\frac32)_{3}},
 &&e^{0,1}_{k,k+2}=\dd\frac{-3(k+1)_{3}}{4(k+4)(k+\frac32)_{4}},\\
 &e^{0,1}_{k,k+3}=\frac{-(k+1)_{3}}{4(k+5)(k+\frac52)_{3}},
 &&e^{0,1}_{k+1,k}=\dd\frac{-3(k+3)_{2}}{4(k+\frac32)_{3}},\\
 &e^{0,1}_{k+2,k}=\frac{-3(k+3)_{5}}{4k(k+\frac32)_{4}},
 &&e^{0,1}_{k+3,k}=\dd\frac{(k+4)_{3}}{4(k+2)(k+\frac52)_{3}}.
\end{align*}
\end{thm}

\section{Dual Petrov-Galerkin algorithms for fifth-order differential equations}
\h We are interested in using the generalized Jacobi-Petrov-Galerkin
method
 to solve the following fifth-order elliptic linear equation
 \bq -u^{(5)}(x)+\al_2\, u^{(4)}(x)+\bt_2\, u^{(3)}(x)-\g_2\, u^{(2)}(x)-\delta_2\, u^{(1)}(x)+\mu_2\,
  u(x)=f(x),\quad x\in (-1,1),\label{5}
 \eq
  subject to the homogeneous boundary conditions
 \begin{equation}\label{5hbc}
u(\pm 1)=u^{(1)}(\pm 1)=u^{(2)}(1)=0.
\end{equation}
We define the following two spaces
$$V=\{{u\in H^{(3)}(I):u(\pm 1)=u^{(1)}(\pm 1)=u^{(2)}(1)=0}\},$$
 and
$$
V^*=\{{u\in H^{(3)}(I):u(\pm 1)=u^{(1)}(\pm 1)=u^{(2)}(-1)=0}\}.$$
where \bq H^{(3)}(I)=\{u:\|u\|_{3,w^{\al,\bt}}<\infty\},\,
\|u\|_{3,w^{\al,\bt}}=\left(\dd\sum_{k=0}^3\|\partial_x^ku\|^2_{w^{\al+k,\bt+k}}\right)^{\frac12}\nonumber\eq
 Let\ $P_{N}$\ be the space of all polynomials of degree less than
 or  equal to $N$. Setting\ $V_{N}=V\cap P_{N}$\ and\ $V_{N}^*=V^*\cap
P_{N}$. We observe that:
$$
V_{N}={\mathrm {span}}\{J^{(-3,-2)}_{5}(x),\,
J^{(-3,-2)}_{6}(x),...,J^{(-3,-2)}_{N}(x)\},
 $$
$$ V^*_{N}=\mathrm{
span}\{J^{(-2,-3)}_{5}(x),J^{(-2,-3)}_{6}(x),...,J^{(-2,-3)}_{N}(x)\}.
$$
The dual Petrov-Galerkin approximation of (\ref{5})-(\ref{5hbc}) is
to find $u_{N}\in V_{N}$ such that
\bq
\begin{split}
&-\left(D^{5}u_{N}(x),v(x)\right)+\al_2\left(D^{4}u_{N}(x),v(x)\right)
+\bt_2\left(D^{3}u_{N}(x),v(x)\right)-\g_2\left(D^{2}u_{N}(x),v(x)\right)\s
 &\q-\delta_2\left(D\,
u_{N}(x),v(x)\right)+\mu_2\left(u_{N}(x),v(x)\right)=(f(x),v(x)),
 \quad\forall v\in V^{*}_N.
 \end{split}
 \eq

{\subsection{The choice of basis functions}} \h We can construct
suitable basis functions and their dual basis by setting \bq
\phi_{k}(x)=J_{k+5}^{(-3,-2)}(x)=(1-x^2)^{2}(1-x)\,
R_{k}^{(3,2)}(x),\quad k=0,1,\ldots,N-5,
  \nonumber\eq
 \bq
 \psi_{k}(x)=J_{k+5}^{(-2,-3)}(x)=(1-x^2)^{2}(1+x)\,
R_{k}^{(2,3)}(x),\quad k=0,1,\ldots,N-5.
  \nonumber\eq\h It is obvious that $\{\phi_{k}(x)\}$\ and\ $\{\psi_{k}(x)\}$\ are
linearly independent. Therefore we have
 $$V_{N}={\mathrm {span}}\{\phi_{k}(x):k=0,1,2,\ldots,N-5\},$$
 and
 $$V_{N}^{*}={\mathrm {span}}\{\psi_{k}(x):k=0,1,2,\ldots,N-5\}.$$
  The following two lemmas are needed.
\begin{lem}
\label{dersy5}
 \begin{equation}
D^{5}J_{k+5}^{(-3,-2)}(x)=\dd-3(k+1)(k+2)(k+4)(k+5)
R_{k}^{(2,3)}(x).\nonumber
 \end{equation}
 \end{lem}
\noindent
\begin{proof} Setting\,$\al=2,\bt=1$ in relation \eqref{6}, we get
\bq(1-x^2)R_{k}^{(3,2)}(x)=\frac{12}{(2k+5)_3}
\left[(k+2)(2k+7)R_{k}^{(2,1)}(x)+2(k+3)R_{k+1}^{(2,1)}(x)
\right.\nonumber\eq\bq\left.-(k+4)(2k+5)R_{k+2}^{(2,1)}(x) \right]
.\nonumber\eq Making use of this relation and with the aid of the
two relations (\ref{dershift}) (for $q=2$) and (\ref{diff3}), we
obtain \bq D^{5}J_{k+5}^{(-3,-2)}(x)=\frac{1}{(2k+5)_{3}}
\left[(2k+7)(k-1)_{7}R_{k-2}^{(3,4)}(x)+2(k)_{7}R_{k-1}^{(3,4)}(x)\right.\nonumber\eq
\bq\left.-(2k+5)(k+1)_{7}R_{k}^{(3,4)}(x) \right].\nonumber\eq
Finally, from the two relations (\ref{shift1}) and (\ref{shift2}),
and after some lengthy manipulation, we get
 \bq
  D^{5}J_{k+5}^{(-3,-2)}(x)
=-3(k+1)(k+2)(k+4)(k+5)R_{k}^{(2,3)}(x).\nonumber\eq
\end{proof}
\begin{lem}
\label{dersy54}
 \begin{eqnarray}&&D^{4}J_{k+5}^{(-3,-2)}(x)=-\frac{3(k+2)(k+4)_3}{2k+7}R_{k+1}^{(2,3)}(x)+
 \frac{3(k+1)_2(k+4)_2}{2(k+\frac52)_{2}}R_{k}^{(2,3)}(x)\\& &+
 \frac{3(k)_{3}(k+4)}{(2k+5)}R_{k-1}^{(2,3)}(x), \nonumber
 \\
 &&D^{3}J_{k+5}^{(-3,-2)}(x)=-\frac{3(k+4)_{4}}{4(k+\frac72)_{2}}R_{k+2}^{(2,3)}(x)+
 \frac{3(k+2)(k+4)_{3}}{2\, (k+\frac52)_{3}}R_{k+1}^{(2,3)}(x)\\&
 &+\frac{3(k+1)_2(k+4)_2}{2(k+\frac{5}{2})_{2}}R_{k}^{(2,3)}(x)-
 \frac{3(k)_{3}(k+4)}{2(k+\frac32)_{3}}R_{k-1}^{(2,3)}(x)-
 \frac{3(k-1)_{4}}{4(k+\frac32)_{2}}R_{k-2}^{(2,3)}(x), \nonumber
 \end{eqnarray}
  \begin{eqnarray}
&&D^{2}J_{k+5}^{(-3,-2)}(x)=-\frac{3(k+4)_{5}}{8(k+3)(k+\frac72)_{3}}R_{k+3}^{(2,3)}(x)+
 \frac{9(k+4)_{4}}{8(k+\frac{5}{2})_{4}}R_{k+2}^{(2,3)}(x)\nonumber\\&
 &+\frac{9(k+2)(k+4)_{3}}{8(k+\frac52)_{3}}R_{k+1}^{(2,3)}(x)-
 \frac{9(k+1)_2(k+4)_2}{4(k+\frac32)_{4}}R_{k}^{(2,3)}(x)\\& &-
 \frac{9(k)_{3}(k+4)}{8(k+\frac32)_{3}}R_{k-1}^{(2,3)}(x)+
 \frac{9(k-1)_{4}}{8(k+\frac12)_{4}}R_{k-2}^{(2,3)}(x)+
 \frac{3(k-2)_{5}}{8(k+3)(k+\frac12)_{3}}R_{k-3}^{(2,3)}(x),
 \nonumber
\\
&&DJ_{k+5}^{(-3,-2)}(x)=-\frac{3(k+5)_{5}}{16(k+3)(k+\frac72)_{4}}R_{k+4}^{(2,3)}(x)+
 \frac{3(k+4)_{5}}{4(k+3)(k+\frac52)_{5}}R_{k+3}^{(2,3)}(x)\nonumber\\&
 &+\frac{3(k+4)_{4}}{4(k+\frac52)_{4}}R_{k+2}^{(2,3)}(x)-
 \frac{9(k+2)(k+4)_{3}}{4(k+\frac32)_{5}}R_{k+1}^{(2,3)}(x)-
 \frac{9(k+1)_2(k+4)_2}{8(k+\frac32)_{4}}R_{k}^{(2,3)}(x)\\& &+
 \frac{9(k)_{3}(k+4)}{4(k+\frac12)_{5}}R_{k-1}^{(2,3)}(x)+
 \frac{3(k-1)_{4}}{4(k+\frac12)_{4}}R_{k-2}^{(2,3)}(x)-
 \frac{3(k-2)_{5}}{4(k+3)(k-\frac12)_{5}}R_{k-3}^{(2,3)}(x)\nonumber\\&
 &-\frac{3(k-3)_{5}}{16(k+3)(k-\frac12)_{4}}R_{k-4}^{(2,3)}(x),\nonumber
\end{eqnarray}
\begin{eqnarray}
and\nonumber\\
&&J_{k+5}^{(-3,-2)}(x)=-\frac{3(k+6)_{5}}{32(k+3)(k+\frac72)_{5}}R_{k+5}^{(2,3)}(x)+
 \frac{15(k+5)_{5}}{32(k+3)(k+\frac52)_{6}}R_{k+4}^{(2,3)}(x)\nonumber\\&
 &+\frac{15(k+4)_{5}}{32(k+3)(k+\frac52)_{5}}R_{k+3}^{(2,3)}(x)-
 \frac{15(k+4)_{4}}{8(k+\frac32)_{6}}R_{k+2}^{(2,3)}(x)-
 \frac{15(k+2)(k+4)_{3}}{16(k+\frac32)_{5}}R_{k+1}^{(2,3)}(x)\\&
 &+\frac{45(k+1)_2(k+4)_2}{16(k+\frac12)_{6}}R_{k}^{(2,3)}(x)+
 \frac{15(k)_{3}(k+4)}{16(k+\frac12)_{5}}R_{k-1}^{(2,3)}(x)-
 \frac{15(k-1)_{4}}{8(k-\frac12)_{6}}R_{k-2}^{(2,3)}(x)\nonumber\\&
 &-\frac{15(k-2)_{5}}{32(k+3)(k-\frac12)_{5}}R_{k-3}^{(2,3)}(x)+
 \frac{15(k-3)_{5}}{32(k+3)(k-\frac32)_{6}}R_{k-4}^{(2,3)}(x)+
 \frac{3(k-4)_{5}}{32(k+3)(k-\frac32)_{5}}R_{k-5}^{(2,3)}(x).\nonumber
 \end{eqnarray}
\end{lem}
 Applying Petrov-Galerkin method to (\ref{5})-(\ref{5hbc}) and if we make use of the two Lemmas \ref{dersy5} and \ref{dersy54}, then after performing some lengthy manipulation, the
 numerical solution of (\ref{5})-(\ref{5hbc}) can be obtained. This solution is given in the following Theorem.
 \begin{thm}
 \label{thm6}
  If
 $u_N(x)=\dd\sum_{0}^{N-5}a_{k}J_{k+5}^{(-3,-2)}(x)$ is the
 Petrov-Galerkin approximation to (\ref{5})-(\ref{5hbc}), then the
 expansion coefficients $\{a_k:k=0,1,\dots,N-5\}$ satisfy the matrix
 system \bq (B_2+\al_2\,G_4+\bt_2\,G_3+\g_2\,G_2+\delta_2\,G_1+\mu_2\,G_0)\mathbf{a}=\mathbf{f^*},
 \label{system2}\eq
 where the nonzero elements of the matrices $B_2$ and $G_i, (0\leqslant i\leqslant4)$
 are given as follows:
\begin{alignat*}{3} &b^2_{kk}=r_{k},
&&g^{4,2}_{kk}=\frac{r_{k}}{2(k+\frac52)_{2}},
 &&g^{4,2}_{k,k+1}=\frac{3(k+1)_{3}(k+5)}{2k+7},\f
 &g^{4,2}_{k+1,k}=\frac{-3(k+2)_{5}}{(k+3)(2k+7)},
  &&g^{3,2}_{kk}=\frac{\,r_{k}}{2(k+\frac{5}{2})_{2}},
&&g^{3,2}_{k,k+1}=\frac{-3(k+1)_{3}(k+5)}{2\, (k+\frac52)_{3}},\f
&g^{3,2}_{k,k+2}=\frac{-3(k+1)_{4}}{4\, (k+\frac72)_{2}},
  && g^{3,2}_{k+1,k}=\frac{3(k+2)(k+4)_{3}}{2\, (k+\frac52)_{3}},
   && g^{3,2}_{k+2,k}=\frac{-3(k+4)_{4}}{4(k+\frac72)_{2}},\f
 \end{alignat*}
 \begin{alignat*}{3}
 &g^{2,2}_{kk}=\frac{3\, r_{k}}{4(k+\frac32)_{4}},
&&g^{2,2}_{k,k+1}=\frac{9(k+1)_{3}(k+5)}{8(k+\frac52)_{3}},
  &&g^{2,2}_{k,k+2}=\frac{-9(k+1)_{4}}{8(k+\frac{5}{2})_{4}},\f
&g^{2,2}_{k,k+3}=\frac{-3\, (k+1)_{5}}{8(k+6)\, (k+\frac72)_{3}},
 &&g^{2,2}_{k+1,k}=\frac{-9\, (k+2)\, (k+4)_{3}}{8\, (k+\frac52)_{3}},
&&g^{2,2}_{k+2,k}=\frac{-9\, (k+4)_{4}}{8\, (k+\frac52)_{4}},\f
&g^{2,2}_{k+3,k}=\frac{3(k+4)_{5}}{8(k+3)\, (k+\frac72)_{3}},
&&g^{1,2}_{kk}=\frac{3\, r_{k}}{8(k+\frac32)_{4}},
&&g^{1,2}_{k,k+1}=\frac{-9(k+1)_{3}(k+5)}{4(k+\frac32)_{5}},\f
&g^{1,2}_{k,k+2}=\frac{-3(k+1)_{4}}{4(k+\frac52)_{4}},
&&g^{1,2}_{k,k+3}=\frac{3(k+1)_{5}}{4(k+6)(k+\frac52)_{5}},
&&g^{1,2}_{k,k+4}=\frac{3(k+1)_{5}}{16(k+7)(k+\frac72)_{4}},\f
&g^{1,2}_{k+1,k}=\frac{9(k+2)(k+4)_{3}}{4(k+\frac32)_{5}},
&&g^{1,2}_{k+2,k}=\frac{-3(k+4)_{4}}{4(k+\frac52)_{4}},
&&g^{1,2}_{k+3,k}=\frac{-3(k+4)_{5}}{4(k+3)(k+\frac52)_{5}},\f
&g^{1,2}_{k+4,k}=\frac{3(k+5)_{5}}{16(k+3)(k+\frac72)_{4}},
&&g^{0,2}_{kk}=\frac{15\, r_{k}}{16(k+\frac12)_{6}},
&&g^{0,2}_{k,k+1}=\frac{15(k+1)_{3}(k+5)}{16(k+\frac32)_{5}},\f
&g^{0,2}_{k,k+2}=\frac{-15(k+1)_{4}}{8(k+\frac32)_{6}},
&&g^{0,2}_{k,k+3}=\frac{-15(k+1)_{5}}{32(k+6)(k+\frac52)_{5}},
&&g^{0,2}_{k,k+4}=\frac{15(k+1)_{5}}{32(k+7)(k+\frac52)_{6}},\f
&g^{0,2}_{k,k+5}=\frac{3(k+1)_{5}}{32(k+8)(k+\frac72)_{5}},\
&&g^{0,2}_{k+1,k}=\frac{-15(k+2)(k+4)_{3}}{16(k+\frac32)_{5}},\
&&g^{0,2}_{k+2,k}=\frac{-15(k+4)_{4}}{8(k+\frac32)_{6}},\f
&g^{0,2}_{k+3,k}=\frac{15(k+4)_{5}}{32(k+3)(k+\frac52)_{5}},\
&&g^{0,2}_{k+4,k}=\frac{15(k+5)_{5}}{32(k+3)(k+\frac52)_{6}},
 &&g^{0,2}_{k+5,k}=\frac{-3(k+6)_{5}}{32(k+3)(k+\frac72)_{5}},
 \end{alignat*}
 where\ $r_{k}=3(k+1)(k+2)(k+4)(k+5).$
\end{thm}
{\section {Structure of the coefficient matrices in the linear
systems (\ref {system1}) and (\ref {system2}) }}
 In this section, we discuss the structure of the
 coefficient matrices $B_1$ and $E_{3-q}\ (1\le q\le 3)$ in the
 linear system (\ref {system1}), and the
 coefficient matrices $B_2$ and $G_{5-q}\ (1\le q\le 5)$ in the
 linear system (\ref {system2}). Hence, we discuss the structure of
 the two combined matrices $D_1=B_1+\dd \al_1\, E_{2}+\dd \bt_1\, E_{1}+\dd \g_1\,
 E_{0}$ and
 $D_2=B_2+\al_2\,G_4+\bt_2\,G_3+\g_2\,G_2+\delta_2\,G_1+\mu_2\,G_0$.
 Also we discuss the influence of these structures on the efficiency
 of the two systems (\ref {system1}) and (\ref {system2}).\s\h
 It is clear that each of the matrices $B_1$ and $B_2$ is diagonal, so it is worthly
 to note that the two cases correspond to $\al_1=\bt_1=\g_1=0$ in (\ref
 {system1}) and
 $\al_2=\bt_2=\g_2=\delta_2=\mu_2=0$ in (\ref {system2}) lead to two
 linear systems with diagonal matrices. The result for these two cases are summarized in the
 following important two corollaries.
 \begin{cor}
If\ $u_{N}(x)=\dd\sum_{k=0}^{N-3}\ a_{k}\
 J_{k+3}^{(-2,-1)}(x)$\ and\ \ $\al_1=\bt_1=\g_1=0$,
 is the Galerkin approximation to problem (\ref{900})-(\ref{1000}), then the expansion
 coefficients\
 $\{a_{k}:\,\ k=0,1,\cdots,N-3\}$ are given explicitly by
  \bq
  a_{k}=\frac{k+2}{16}f_{k},\ k=0,1,\cdots,N-3,\label{diagonal}
  \nonumber\eq
  where $f_{k}=\dd\int\limits^1_{-1}(1-x^2)(1+x)f(x)R_{k}^{(1,2)}(x)dx.$
\end{cor}
\begin{cor}
If\ $u_{N}(x)=\dd\sum_{k=0}^{N-5}\ a_{k}\
 J_{k+5}^{(-3,-2)}(x)$\ and\ \ $\al_2=\bt_2=\g_2=\delta_2=\mu_2=0$,
 is the Petrov-Galerkin approximation to
 the problem (\ref{5})-(\ref{5hbc}), then the expansion
 coefficients\\
 $\{a_{k}:k=0,1,\cdots,N-5\}$ are given explicitly by
  \bq
  a_{k}=\frac{k+3}{384}f_{k},\ k=0,1,\ldots,N-5,\label{diagonal5}
  \nonumber\eq
  where $f_{k}=\dd\int\limits^1_{-1}(1-x^2)^2(1+x)f(x)R_{k}^{(2,3)}(x)dx.$
\end{cor}

Now, each of the matrices $E_{3-q}\ (1\le q\le 3)$ and  $G_{5-q}\
(1\le q\le 5)$  is a band matrix whose total number of nonzero
diagonals upper or lower the main diagonal is $q$. Thus the
coefficient matrices\ $D_1$\ and $D_2$\  are four-band and six-band
matrices, respectively at most. These special structures of\ $D_1$
and $D_2$ simplify greatly the solution of the two linear systems
(\ref{system1}) and (\ref{system2}). These two systems can be
factorized by $LU$-decomposition and the number of operations
necessary to construct these factorizations are of order\ $21(N-2)$
and $55(N-4)$ respectively, and the number of operations needed to
solve the two triangular systems are of order $13(N-2)$ and
$21(N-4)$ respectively.\s
  {\bf {Note.} } The total number of operations mentioned in the previous
  discussion includes the number of all subtractions, additions, divisions and
  multiplications. (see, \cite{Sch}).

\section{Nonhomogeneous boundary
 conditions}
\h In the following we describe how third- and fifth-order problems
with nonhomogeneous boundary conditions can be
transformed into problems with homogeneous boundary conditions.\\
Let us consider the one-dimensional third-order equation \bq
u^{(3)}(x)-\alpha_1\, u^{(2)}(x)-\beta_1\, u^{(1)}(x)+\gamma_1\,
u(x)=f(x),\quad x\in I=(-1,1),
 \nonumber\eq
 subject to the nonhomogeneous boundary conditions:
 \bq u(\pm1)=a_{\pm},\quad u^{(1)}(1)=a^{1}.
\label{non0}\eq In such case we proceed as follows: \s Set \bq
V(x)=u(x)+a_{0}+a_{1}x+a_{2}x^{2},\label{nonhom}\eq where
\begin{eqnarray*}
\dd a_{0}&=&\frac{-a_{-}-3a_{+}+2a^{1}}{4},\\
\dd a_{1}&=&\frac{a_{-}-a_{+}}{2},\\
\dd a_{2}&=&\frac{-a_{-}+a_{+}-2a^{1}}{4}.
\end{eqnarray*}
The transformation (\ref{nonhom}) turns the nonhomogeneous boundary
conditions
(\ref{non0}) into the homogeneous boundary conditions\\
\bq V(\pm1)=V^{(1)}(1)=0.\label{hom}\eq Hence it suffices to solve
the following modified one-dimensional third-order equation:
 \bq \label{3r3}
V^{(3)}(x)-\alpha_1\, V^{(2)}(x)-\beta_1\, V^{(1)}(x)+\gamma_1\,
V(x)=f^{*}(x),\quad x\in I=(-1,1), \eq subject to the homogeneous
boundary conditions (\ref{hom}), where $V(x)$ is given by
(\ref{nonhom}), and\\ $
f^{*}(x)=f(x)+(-2\al_1\,a_{2}-\bt_1\,a_{1}+\g_1\,a_{0})+(-2\bt_1\,a_{2}+\g_1\,a_{0})x+\g_1\,a_{2}x^{2}.$
\s If we apply the Petrov-Galerkin method to the modified equation
(\ref{3r3}), we get the equivalent system of equations
 \bq
(B_1+\al_1\,E_2+\bt_1\,E_1+\g_1\,E_0)\textbf{a}=\textbf{f}^*,\nonumber\eq
 where
$B_1,E_2,E_1$ and $E_0$ are the matrices defined in Theorem \ref{thm4}, and
$\textbf{f}^*=(f_0^*,f_1^*,\dots,f_{N-3}^*)$,\\
and

\[f^*_k =\begin{cases}
  -2\al_1\,a_{2}-\bt_1\,a_{1}+\g_1\,a_{0},\quad &k = 0, \\
  \frac{6}{5}(-2\bt_1\,a_{2}+\g_1\,a_{0}),\quad &k = 1, \\
   \frac{10}{7}\g_1\,a_{2},\quad &k = 2, \\
f_k, \quad & k\geqslant3,\\
  \end{cases}\] where, $f_k=\dd\int_{-1}^1(1-x^2)(1+x)\,R_k^{(1,2)}(x)\,f(x)\,dx$.\s
  \h We can apply the same procedures to solve the fifth-order
  equation \bq
-u^{(5)}(x)+\al_2\, u^{(4)}(x)+\bt_2\, u^{(3)}(x)-\g_2\,
u^{(2)}(x)-\delta_2\, u^{(1)}(x)+\mu_2\,
  u(x)=f(x),\quad x\in (-1,1),\label{5non}
 \eq subject to the nonhomogeneous boundary conditions
 \bq\label{5nhbc}
u(\pm 1)=a_{\pm},\quad u^{(1)}(\pm 1)=\overset{1}{a_{\pm}},\quad
u^{(2)}(1)=\overset{2}{a_+}. \eq In such case,
(\ref{5non})-(\ref{5nhbc}) can be transformed into \bq
-V^{(5)}(x)+\alpha_2\, V^{(4)}(x)+\beta_2\, V^{(3)}(x)-\gamma_2\,
V^{(2)}(x)-\delta_2\, V^{(1)}(x)+\mu_2\, V(x)=f^{*}(x),\quad x\in
I=(-1,1),
 \label{5hom}\eq
subject to the homogenous boundary conditions \bq
V(\pm1)=V^{(1)}(\pm1)=V^{(2)}(1)=0,\label{homg5}\nonumber\eq
where
\bq
V(x)=u(x)+a_{0}+a_{1}x+a_{2}x^{2}+a_{3}x^{3}+a_{4}x^{4},\label{t}\nonumber\eq\quad
with
\begin{eqnarray*}
\dd a_{0}&=&\frac{1}{16} \left(-2 \overset{1}{a_-}+8 \overset{1}{a_+}-2 \overset{2}{a_+}-5 a_{-}-11 a_+\right),\\
\dd a_{1}&=&\frac{1}{4} \left(\overset{1}{a_-}+\overset{1}{a_+}+3 a_{-}-3 a_+\right),\\
\dd a_{2}&=&\frac{1}{8} \left(-6 \overset{1}{a_+}+2 \overset{2}{a_+}-3 a_-+3 a_+\right),\\
\dd a_{3}&=&\frac{1}{4} \left(-\overset{1}{a_-}-\overset{1}{a_+}-a_-+a_+\right),\\
\dd a_{4}&=&\frac{1}{16} \left(-2 \overset{2}{a_+}+4
\overset{1}{a_+}+2 \overset{2}{a_-}+3 a_{-}-3 a_+\right),
\end{eqnarray*}
and \begin{eqnarray*}
f^{*}(x)\dd=&&(\mu_2\,a_{0}-\delta_2\,a_{1}-2\g_2\, a_{2}+6\bt\,
a_{3}+24\al_2\, a_{4})+ (\mu_2\, a_{1}-2\delta_2\, a_{2}-6\g_2\,
a_{3}+24\bt_2\, a_{4})x\\\dd& &+(\mu_2\, a_{2}-3\delta_2\,
a_{3}-12\g_2\, a_{4})x^{2}+(\mu_2\, a_{3}-4\delta_2\,
a_{4})x^{3}+\mu_2\, a_{4}x^{4}+f(x).\end{eqnarray*}
 If we apply the Petrov-Galerkin method to the modified equation
(\ref{5hom}), we get the equivalent system of equations \bq
(B_2+\al\,G_4+\bt\,G_3+\g\,G_2+\delta\,G_1+\mu\,G_0)\textbf{a}=
\textbf{f}^*,\nonumber\eq where $B_2,G_i,0\leqslant i\leqslant4$ are
the matrices defined in Theorem \ref{thm6}, and
$\textbf{f}^*=(f_0^*,f_1^*,\dots,f_{N-5}^*)$,
 \bq f_k^* =
\begin{cases}
   \mu_2\,a_{0}-\delta_2\,a_{1}-2\g_2\, a_{2}+6\bt_2\,
a_{3}+24\al_2\, a_{4},\quad &k = 0, \\
  \frac{8}{7}(\mu_2\, a_{1}-2\delta_2\, a_{2}-6\g_2\,
a_{3}+24\bt_2\, a_{4}),\quad &k = 1, \\
   \frac{4}{3}(\mu_2\, a_{2}-3\delta_2\, a_{3}-12\g_2\,
a_{4}),\quad &k = 2, \\
 \frac{50}{33}(\mu_2\, a_{3}-4\delta_2\, a_{4}),\quad &k = 3, \\
    \frac{238}{143}\mu_2\,
a_{4}, \quad &k= 4, \\
f_k, \quad & k\geqslant5,\\
  \end{cases}
\nonumber\eq where,
$f_k=\dd\int_{-1}^1(1-x^2)^2(1+x)\,R_k^{(2,3)}(x)\,f(x)\,dx$.

\section{Condition number of the resulting matrices} \hspace*{16pt}For the direct
collocation method, the condition numbers behave like\ $O(N^{6})$
and $O(N^{10})$\ for third- and fifth-order respectively ($N$:
maximal degree of polynomials). In this paper we obtain improved
condition numbers with\ $O(N^{4})$ and $O(N^{6})$ respectively for
third- and fifth-order. The advantages with respect to propagation
of rounding errors is demonstrated.\s
 \h  For GJPGM, the
resulting systems from the two differential equations\
$u^{(3)}(x)=f(x)$ and $-u^{(5)}(x)=f(x)$\ are \ $B_{1}\,
{\mathbf{a^{1}}}={\mathbf{f^*}}$\ and \ $B_{2}\,
{\mathbf{a^{2}}}={\mathbf{f^*}}$, where \ $B_{1}$\ and \ $B_{2}$ are
two diagonal matrices whose diagonal elements are given by
$b^1_{kk}$ and $b^2_{kk}$, where
$$b^1_{kk}=2(k+1)(k+3),\quad b^2_{kk}=3(k+1)(k+2)(k+4)(k+5).$$
Thus we note that the condition numbers of the matrices\ $B_{1}$ and
$B_{2}$\ behave like $O(k^{2})$ and $O(k^{4})$\ respectively for
large values of\ $k$. The evaluation of the condition numbers for
the matrices\ $B_{1}$ and $B_{2}$ \ are easy because of the special
structure of
 them, since $B_{1}$ and $B_{2}$\ are diagonal matrices, so their eigenvalues are their diagonal
elements, and the condition number in such case has the definition\\[0.2cm]
 $$\mathrm {Condition\  number\ of \ the \
matrix=\dd \frac{Max\ (eigenvalue\  of\  the\  matrix)} {Min\
(eigenvalue\  of\  the\  matrix)}}.$$
\h In Table 1 we list the
values of the conditions numbers for the matrices\
 $B_1$ and $B_{2}$, respectively, for different values of \ $N$.
\vspace{50pt}
 \begin{center}
  \begin{center}
Table\  1\\ Condition number for the matrix \ $B_{n},n=1,2$\\[0.1cm]
\end{center}
\begin{center}
\begin{tabular}
{|c|c|c|c|c|c|c|}\hline $n$&$N$&$\alpha
_{min}$&$\alpha_{max}$&Cond($B_{n}$)&Cond($B_{n}$)/$N^{2n}$\\\hline
&16&&448&74.667&2.917\ .\ $10^{-1}$\\[0.1cm]
&20&&720&120&3.000\ .\ $10^{-1}$\\[0.1cm]
&24&&1056&176&3.056\ .\
$10^{-1}$\\[0.1cm]
 1&28&6&1456&242.667&3.095\ .\ $10^{-1}$\\[0.1cm]
&32&&1920&320&3.125\ .\ $10^{-1}$\\[0.1cm]
&36&&2448&408&3.148\ .\ $10^{-1}$\\[0.1cm]
 &40&&3040&506.667&3.167\ .\ $10^{-1}$\\[0.1cm]
\hline &16&&112320&936&1.428\ .\ $10^{-2}$\\[0.1cm]
&20&&310080&120&1.615\ .\ $10^{-2}$\\[0.1cm]
&24&&695520&5796&1.747\ .\
$10^{-2}$\\[0.1cm]
 2&28&120&1.361$\ .\ 10^{6}$&11340&1.845$\ .\ 10^{-2}$\\[0.1cm]
&32&&2.416$\ .\ 10^{6}$&20137.6&1.920\ .\ $10^{-2}$\\[0.1cm]
&36&&$3.992\ .\ 10^{6}$&33264&$1.981\ .\ 10^{-2}$\\[0.1cm]
 &40&&6.234\ .\ $10^{6}$&51984&2.029\ .\ $10^{-2}$\\
\hline
\end{tabular}
\end{center}
\end{center}
 {\bf {Remark 1.}} If we add
 $\dd\sum_{q=1}^{3}E_{3-q}\ (1\le q\le 3)$ and \ $\dd\sum_{q=1}^{5}G_{5-q}\ (1\le q\le 5)$\ , where
 the matrices\ $E_{3-q}$\ and\ $G_{5-q}$\ are the matrices their nonzero elements are given
 explicitly in Theorems \ref{thm4} and \ref{thm6} respectively,
 to the matrices $B_1$ and $B_2$ respectively, then we find that the
eigenvalues of matrices $D_{1}=
B_{1}+\dd\sum_{q=1}^{3}E_{3-q}$,\quad$D_{2}=
B_{2}+\dd\sum_{q=1}^{5}G_{5-q}$ are all real positive. Moreover, the
effect of these additions does not significantly change the values
of the condition numbers for the systems. This means that matrices
$B_1$ and $B_2$, which resulted from the highest derivatives of the
differential equations under investigation, play the most important
role in the propagation of the roundoff errors. The numerical
results of Table 2 illustrate this remark.\\
 \begin{center}
\begin{center}
Table\  2\\ Condition number for the matrix \ $D_{n},n=1,2$\\[0.1cm]
\begin{center}
\begin{tabular}
{|c|c|c|c|c|c|c|c|c|c|}\hline
$N$&Cond($D_{1}$)&$\mathrm{\dd\frac{Cond(D_{1})}{N^{2}}}$&Cond($D_{2}$)&$\mathrm{\dd\frac{Cond(D_{2})}{N^{4}}}$
 \f\hline
 &&&&\\[-0.3cm]
 16&55.287&2.159\ . $10^{-1}$&827.262&1.262\ . $10^{-2}$\f
 20&88.679&2.217\ . $10^{-1}$&2278.4&1.424\ . $10^{-2}$\f
  24&129.929&2.256\ . $10^{-1}$&5104.45&1.539\ . $10^{-2}$\f
   28&179.037&2.284\ . $10^{-1}$&9980.18&1.624\ . $10^{-2}$\f
    32&236.003&2.305\ . $10^{-1}$&17715.3&1.689\ . $10^{-2}$\f
     36&300.826&2.321\ . $10^{-1}$&2925.4&1.742\ . $10^{-2}$\f
      40&373.507&2.334\ . $10^{-1}$&45677.4&1.784\ . $10^{-2}$\f
 \hline
\end{tabular}
\end{center}
\end{center}
\end{center}

\newpage
\section{Numerical results}

{\bf Example 1.}\s Consider the one dimensional  equation \bq
u^{(3)}(x)-\al_1\,u^{(2)}(x)-\bt_1\,u^{(1)}(x)+\g_1\,u(x)=f(x),\qquad
u(\pm 1)=u^{(1)}(1)=0,\label{numeriacal1}\eq where\, $f(x)$\,is
chosen such that the exact solution for (\ref{numeriacal1})
is\,$u(x)=(1-x^2)\,x^j\,\sin(m\,\pi\,x)$, $j,m\in\mathbb{N}$. We
have\, $u_{N}(x)=\dd\sum_{k=0}^{N-3}a_k
(1-x^2)(1-x)R_{k}^{(2,1)}(x)$ and the vector of unknowns\\
$\textbf{a}=(a_0,a_1,\dots,a_{N-3})^T$ is the solution of the system
\q$(B_1+\dd \al_1\, E_{2}+\dd \bt_1\, E_{1}+\dd \g_1\, E_{0})\,
\mathbf{a}=\mathbf{f^*},$ \s
 where the nonzero elements of the
matrices\ $B_1$\ and\  $E_{i,n}\ (0\le i\le 2)$\ are given
explicitly in Theorem \ref{thm6}.\s
 Table 3 lists the maximum pointwise
error $E$ for $u-u_N$ to (\ref{numeriacal1}), using GJPGM for
various values of $j,m$ and the coefficients $\al_1,\bt_1$ and
$\g_1.$\\
\begin{center}
\begin{center} Table\  3\\ {Maximum pointwise error for\ $u-u_N$\
for\ $N=8,12,16,20,24$}
\begin{center}\begin{tabular}
{|c|c|c|c|c|c|c|c|c|c|c|}\hline
N&$j$&$m$&$\al_1$&$\bt_1$&$\g_1$&$E$&$\al_1$&$\bt_1$&$\g_1$&$E$\\
\hline
 8&&&&&&$2.558\ .\ 10^{-3\,\,\,}$&8&$8^2$&$8^3$&$2.872\ .\ 10^{-3\,\,\,}$\\
12&&&&&&$1.909\ .\ 10^{-6\,\,\,}$&12&$12^2$&$12^3$&$2.224\ .\ 10^{-6\,\,\,}$\\
16&1&1&0&0&0&$4.368\ .\ 10^{-10}$&16&$16^2$&$16^3$&$4.122\ .\ 10^{-10}$\\
20&&&&&&$2.811\ .\ 10^{-14}$&20&$20^2$&$20^3$&$2.961\ .\ 10^{-14}$\\
24&&&&&&$3.885\ .\ 10^{-16}$&24&$24^2$&$24^3$&$2.220\ .\ 10^{-16}$\\
 \hline
8&&&&&&$4.472\ .\ 10^{-3\,\,\,}$&$8^3$&$8^2$&8&$9.409\ .\ 10^{-3\,\,\,}$\\
12&&&&&&$3.687\ .\ 10^{-6\,\,\,}$&$12^3$&$12^2$&12&$8.399\ .\ 10^{-6\,\,\,}$\\
16&0&1&2&3&4&$6.660\ .\ 10^{-10}$&$16^3$&$16^2$&16&$2.178\ .\ 10^{-9\,\,\,}$\\
20&&&&&&$4.529\ .\ 10^{-14}$&$20^3$&$20^2$&20&$1.455\ .\ 10^{-13}$\\
24&&&&&&$7.771\ .\ 10^{-16}$&$24^3$&$24^2$&24&$6.106\ .\ 10^{-16}$\\
 \hline
8&&&&&&$1.119\ .\ 10^{-1\,\,\,}$&8&$8^2$&$8^3$&$1.341\ .\ 10^{-1\,\,\,}$\\
12&&&&&&$2.060\ .\ 10^{-3\,\,\,}$&12&$12^2$&$12^3$&$2.430\ .\ 10^{-3\,\,\,}$\\
16&1&2&0&1&0&$8.934\ .\ 10^{-6\,\,\,}$&16&$16^2$&$16^3$&$8.459\ .\ 10^{-6\,\,\,}$\\
20&&&&&&$1.009\ .\ 10^{-8}$&20&$20^2$&$20^3$&$1.072\ .\ 10^{-8\,\,\,}$\\
24&&&&&&$4.156\ .\ 10^{-12}$&24&$24^2$&$24^3$&$4.746\ .\ 10^{-12}$\\
 \hline
8&&&&&&$1.578\ .\ 10^{-2\,\,\,}$&$8^3$&$8^2$&8&$3.927\ .\ 10^{-1\,\,\,}$\\
12&&&&&&$3.749\ .\ 10^{-5\,\,\,}$&$12^3$&$12^2$&12&$8.773\ .\ 10^{-3\,\,\,}$\\
16&2&1&1&0&1&$1.324\ .\ 10^{-8\,\,\,}$&$16^3$&$16^2$&16&$4.369\ .\ 10^{-5\,\,\,}$\\
20&&&&&&$1.539\ .\ 10^{-12}$&$20^3$&$20^2$&20&$5.206\ .\ 10^{-8\,\,\,}$\\
24&&&&&&$2.498\ .\ 10^{-16}$&$24^3$&$24^2$&24&$2.417\ .\ 10^{-11}$\\
\hline
\end{tabular}
\end{center}\end{center}
\end{center}
{\bf Example 2.}\s Consider the one dimensional fifth-order equation
\bq
-u^{(5)}(x)+\al_2\,u^{(4)}(x)+\bt_2\,u^{(3)}(x)-\g_2\,u^{(2)}(x)-\delta_2\,u^{(1)}(x)+\mu_2\,u(x)=f(x),\nonumber\eq\bq
u(\pm 1)=u^{(1)}(\pm 1)=u^{(2)}(1)=0,\label{numerical2}\eq where\,
$f(x)$ is chosen such that the exact solution for (\ref{numerical2})
is $u(x)=(1-x^2)^2(1-x)\,\cosh(m\,x),$ $m\in\mathbb{R}$. We have\,
$u_{N}(x)=\dd\sum_{k=0}^{N-5}a_k
(1-x^2)^2(1-x)R_{k}^{(3,2)}(x)$ and the vector of unknowns\\
$\textbf{a}=(a_0,a_1,\dots,a_{N-5})^T$ is the solution of the system
$$(B_2+\al_2\,G_4+\bt_2\,G_3+\g_2\,G_2+\delta_2\,G_1+\mu_2\,G_0)\textbf{a}=
\textbf{f}^*,$$
 where the nonzero elements of the
matrices\ $B_2$\ and\  $G_{i,n}\ (0\le i\le 4)$\ are given
explicitly in Theorem \ref{thm6}.\s
 Table 4 lists the maximum pointwise error $E$ for
$u-u_N$ to (\ref{numerical2}), using GJPGM for various values of $m$
and the coefficients $\al_2,\bt_2$ and $\g_2$ and $\delta_2$ and
$\mu_2.$\\
\begin{center}
\begin{center}Table\  4\\{Maximum pointwise error for\ $u-u_N$\\
for $N=8,12,16,20,24$}
\begin{center}\begin{tabular} {|c|c|c|c|c|c|c|c|}\hline
N&$m$&$\al_2$&$\bt_2$&$\g_2$&$\delta_2$&$\mu_2$&$E$\\
\hline
 8&&&&&&&$1.135\ .\ 10^{-1\,\,\,}$\\
12&&&&&&&$2.464\ .\ 10^{-4\,\,\,}$\\
16&3&0&0&0&0&0&$8.165\ .\ 10^{-8\,\,\,}$\\
20&&&&&&&$1.098\ .\ 10^{-11}$\\
24&&&&&&&$5.551\ .\ 10^{-16}$\\
 \hline
8&&&&&&&$1.102\ .\ 10^{-3\,\,\,}$\\
12&&&&&&&$3.164\ .\ 10^{-8\,\,\,}$\\
16&1&1&1&1&1&1&$1.312\ .\ 10^{-13}$\\
20&&&&&&&$2.220\ .\ 10^{-16}$\\
24&&&&&&&$2.220\ .\ 10^{-16}$\\
 \hline
8&&&&&&&$1.927\ .\ 10^{-2\,\,\,}$\\
12&&&&&&&$8.652\ .\ 10^{-6\,\,\,}$\\
16&2&0&1&0&1&0&$5.776\ .\ 10^{-10}$\\
20&&&&&&&$1.598\ .\ 10^{-14}$\\
24&&&&&&&$3.330\ .\ 10^{-16}$\\
 \hline
 8&&&&&&&$6.658\ .\ 10^{-5\,\,\,}$\\
12&&&&&&&$1.215\ .\ 10^{-10}$\\
16&$\frac{1}{2}$&1&2&1&2&1&$6.661\ .\ 10^{-16}$\\
20&&&&&&&$6.661\ .\ 10^{-16}$\\
24&&&&&&&$6.661\ .\ 10^{-16}$\\
 \hline
\end{tabular}
\end{center}\end{center}
\end{center}
{\bf Example 3.}\s Consider the one dimensional nonhomogeneous
equation \bq
u^{(3)}(x)-\al_1\,u^{(2)}(x)-\bt_1\,u^{(1)}(x)+\g_1\,u(x)=f(x),\nonumber\eq\bq
u(\pm 1)=\pm \sinh(m),\quad u^{(1)}(1)=m\,\cosh(m),\quad
m\in\mathbb{R},\label{numerical3}\eq where\, $f(x)$ is chosen such
that
the exact solution for (\ref{numerical3}) is\quad $u(x)=\sinh(m\,x)$.\\
setting \bq
V(x)=u(x)-\sinh(m)\,x+\frac{1}{2}\big[m\,\cosh(m)-\sinh(m)\big]\,(1-x^2),\nonumber\eq
then the differential equation (\ref{numerical3}) is equivalent to
the differential equation\bq\,u^{(3)}(x)\,-\al_1\,
u^{(2)}(x)-\bt_1\, u^{(1)}(x)+\g_1\, u(x)=f(x),\qquad x\in
(-1,1),\quad u(\pm 1)=u^{(1)}(1)=0.
 \nonumber\eq  In Table 5 we list the maximum pointwise
error $E$ for $u-u_N$ to (\ref{numerical3}), using GJPGM for various
values of $m$ and the coefficients $\al_1,\bt_1$ and $\g_1$.
\begin{center}
\begin{center}Table\  5\\{Maximum pointwise error for\ $u-u_N$\\
for $N=8,12,16$}

\begin{center}\begin{tabular} {|c|c|c|c|c|c|}\hline
N&$m$&$\al_1$&$\bt_1$&$\g_1$&$E$\\
\hline
 8&&&&&$2.804\ .\ 10^{-8\,\,\,}$\\
12&1&0&0&0&$9.536\ .\ 10^{-14}$\\
16&&&&&$1.110\ .\ 10^{-16}$\\
 \hline
8&&&&&$2.819\ .\ 10^{-8\,\,\,}$\\
12&1&1&1&1&$9.736\ .\ 10^{-14}$\\
16&&&&&$1.110\ .\ 10^{-16}$\\
 \hline
8&&&&&$1545\ .\ 10^{-5\,\,\,}$\\
12&2&0&1&0&$8.248\ .\ 10^{-10}$\\
16&&&&&$1.310\ .\ 10^{-14}$\\
 \hline
 8&&&&&$6.919\ .\ 10^{-4\,\,\,}$\\
12&3&1&0&1&$1.808\ .\ 10^{-7\,\,\,}$\\
16&&&&&$1.414\ .\ 10^{-11}$\\
 \hline
\end{tabular}
\end{center}\end{center}
\end{center}


\section{Concluding remarks}
\h In this paper, an algorithm for obtaining a numerical spectral
solution for third- and fifth-order differential equations using
certain nonsymmetric generalized Jacobi-Galerkin method is
discussed. The algorithms are very efficient. We have found that,
our choice for a certain family of basis functions to solve third-
and fifth-order differential equations always lead to linear systems
with band matrices that can be efficiently inverted. These special
structures, of course simplifies greatly the numerical computations.
In particular, for some particular third- and fifth-order
differential equations, the resulting systems of these equations are
diagonal. high accurate approximate solutions are achieved using a
few number of the generalized Jacobi polynomials. The obtained
numerical results are comparing favorably with  the analytical ones.
Furthermore, we do believe that the proposed technique can be
applied to Korteweg-de Vries (KDV) equations.

\end{document}